\documentclass[12pt]{amsart}
\usepackage{amsmath,amsthm,amscd,amsfonts,amssymb,graphicx,color,MnSymbol,enumerate,tikz-cd}
\usepackage{hyperref,cleveref}
\usepackage{xcolor}
\usepackage[labelformat=empty]{caption,subcaption}
\newtheorem{theorem}{Theorem}[section]
\newtheorem{lemma}[theorem]{Lemma}
\newtheorem{proposition}[theorem]{Proposition}
\newtheorem{corollary}[theorem]{Corollary}
\theoremstyle{definition}
\newtheorem{definition}[theorem]{Definition}

\theoremstyle{remark}
\newtheorem{remark}[theorem]{Remark}

\numberwithin{equation}{section}

\allowdisplaybreaks

\begin{document}
\baselineskip=17pt
	
	\title[On Chains Associated with Abstract Key Polynomials]{On Chains Associated with Abstract Key Polynomials}
	\author[Sneha Mavi]{Sneha Mavi}
	\address{Department of Mathematics\\ University of Delhi\\  Delhi-110007, India.}
	\email{mavisneha@gmail.com}
	\author[Anuj Bishnoi]{Anuj Bishnoi$^\ast$}
	\address{Department of Mathematics\\  University of Delhi \\   Delhi-110007, India.}
	\email{abishnoi@maths.du.ac.in}
	
	\begin{abstract}
	In this paper, 	for a henselian  valued field $(K,v)$ of arbitrary rank  and an extension $w$ of $v$ to $K(X),$  we use abstract key polynomials for $w$ to give a connection between complete sets, saturated distinguished chains and Okutsu frames. Further, for a valued field $(K,v),$ we also obtain a close connection between complete set of ABKPs for   $w$   and Maclane-Vaqui\'e chains  of $w.$ 
	\end{abstract}
	\subjclass[2020]{12F20, 12J10,   13A18}
	\keywords{Abstract key polynomials, key polynomials, minimal pairs,  MacLane-Vaqui\'e  chains,  Okutsu frames, optimal Maclane chains, saturated distinguished chains}
	\thanks{$^\ast$Corresponding author, E-mail address: abishnoi@maths.du.ac.in}
	\maketitle
	
	\section{Introduction }
	Let $(K,v)$ be a henselian valued field and $w$ be an extension of $v$ to $K(X).$ In this paper, we first  give some  conditions under which   a saturated distinguished chain leads to the notion of a complete set of abstract key polynomials for a valuation-transcendental extension $w.$ Recall that a   valuation-transcendental extension of  $v$ to $K(X)$ is either value-transcendental or residually transcendental   and they  are well studied using abstract key polynomials (see \cite{MMS}, \cite{NS}-\cite{NSD}). 
	
	 In 1982, Okutsu associated to a monic irreducible polynomial $F\in K[X]$ a
	family of monic irreducible polynomials, $F_1, . . . , F_r,$  called the primitive divisor polynomials of $F$ \cite{Oku}, later these polynomials  were  studied   in papers \cite{GME}, \cite{JS},  \cite{Ol-Na} and  \cite{Ol}, and they called the chain  of such polynomials $[F_1, . . . , F_r],$  an Okutsu frame for $F.$ Moreover, they proved that Okutsu frames,   saturated distinguished chains and  optimal Maclane chains  are closely related. In this paper, we also establish a similar connection between  saturated distinguished chains and Okutsu frames, however, our proof is elementary. 
	
	 Next, for a valued field $(K,v),$ we give some  conditions under which a  complete set of ABKPs for a valuation $w$ of  $K(X)$ give rise to an  optimal Maclane chain of $w$ and conversely.     It is also observed that over a residually transcendental extension, the notion of saturated distinguished chains,  Okutsu frames, optimal Maclane chains and  complete set of ABKPs (under certain conditions) are equivalent.
	 
	 In  1936,   Maclane \cite{M}, proved that an extension $w$ of a discrete rank one valuation  $v$ to $K[X]$ can be obtained as a chain of augmentations
	 \begin{align*}
	 	w_0\xrightarrow{\phi_1,\gamma_1} w_1\xrightarrow{\phi_2,\gamma_2}\cdots \longrightarrow w_{n-1}\xrightarrow{\phi_n,\gamma_n} w_n\longrightarrow\cdots
	 \end{align*}
for some suitable  key polynomials $\phi_i\in K[X]$ for  intermediate valuations and elements  $\gamma_{i}$ in some  totally ordered abelian group containing the value group of $v$ as an ordered subgroup. Later,   Vaqui\'e generalized Maclane's theory to arbitrary valued fields (see \cite{V}). More
	recently,  Nart gave a survey of generalized Maclane-Vaqui\'e  theory   in \cite{EN} and \cite{EN1}. 
	Starting with a valuation $w_0$ which admit key polynomials of degree one, Nart also introduced Maclane-Vaqui\'e  chains, consisting of a mixture of ordinary and limit augmentations satisfying some conditions (see Definition \ref{1.1.13}). The main result Theorem 4.3  of \cite{EN1} says that all extensions  $w$  of $v$ to $K[X]$ falls exactly in one of the following category:
	\begin{enumerate}[(i)]
		\item It is the last valuation of a complete finite Maclane-Vaqui\'e chain, i.e., after a finite number $r$ of augmentation steps, we get $w_r=w.$
		\item After a finite number $r$ of augmentation steps, it is the stable limit of a continuous family of augmentations  of $w_r$ defined by key polynomials of constant degree.
		\item  It is the stable limit of a complete infinite Maclane-Vaqui\'e chain.
	\end{enumerate}
	    In this paper, we  study  Maclane-Vaqui\'e chains of the first type and prove that a precise complete finite  Maclane-Vaqui\'e chain can be obtained   using a given complete set $\{Q_i\}_{i\in\Delta}$ of ABKPs for $w$ such that $\Delta$ has a maximal element. Conversely, if $w$ is the last valuation of some complete finite   Maclane-Vaqui\'e chain, then there exists a complete set $\{Q_i\}_{i\in\Delta}$ of ABKPs for $w$ such that $\Delta$ has a maximal element.
	
	To state the main results of the paper,  we first recall some notations, definitions and  preliminary results.

	\section{Notations, Definitions and Statements of Main Results}
	
	Throughout the paper, $(K,v)$ denote a    valued field of arbitrary rank  with value group $\Gamma_ v,$ valuation ring $O_v$ having a unique maximal ideal $M_v,$ and    residue field  $k_{v}.$ Let $\bar{v}$ be an extension of $v$ to a fixed algebraic closure $\overline{K}$ of $K$ with value group $\Gamma_{\bar{v}}$ and residue field $k_{\bar{v}}.$   Let $w$ be  an extension of $v$ to the simple transcendental extension $K(X)$ of $K$ with value group $\Gamma_{w}$ and residue field $k_{w}.$

	\medskip
 All extensions  of $v$ to $K(X)$ are classified  as follows:
	\begin{definition}
		The extension $w$ of $v$ to $K(X)$ is said to be \emph{valuation-algebraic} if $\frac{\Gamma_{w}}{\Gamma_{v}}$ is a torsion group and $k_w$ is algebraic over $k_v.$ The extension $w$ is said to be \emph{value-transcendental} if $\frac{\Gamma_w}{\Gamma_v}$ is a torsion free group and $k_w$ is algebraic over $k_v.$ 
	\end{definition}
	If $L$ is an extension field of $K,$ then an extension $v_L$ of $v$ to $L$ is called residually transcendental (abbreviated as r.\ t.) if the corresponding residue field extension $k_{v_L}/ k_v$ is transcendental.
	
	\begin{definition}
		The extension $w$ of $v$ to $K(X)$ is called \emph{valuation-transcendental} if $w$ is either value-transcendental or is residually transcendental.
	\end{definition}
	An extension $\overline{w}$ of $w$ to $\overline{K}(X)$ which is also an extension of $\bar{v}$ is called a \emph{common extension} of $w$ and $\bar{v}.$
	A common extension $\overline{w}$  is valuation-transcendental if and only if $w$ is valuation-transcendental (see \cite[Lemma 3.3]{FV-K}).

	Let $\overline{w}$ be a  common extension of  $w$ and $\bar{v}$ to $\overline{K}(X).$ Then
	for a pair $(\alpha,\delta)\in\overline{K}\times\Gamma_{\overline{w}},$ the map $\overline{w}_{\alpha,\delta}: \overline{K}[X]\longrightarrow \Gamma_{\overline{w}},$ given by $$\overline{w}_{\alpha,\delta}\left(\sum_{i\geq 0} c_i (X-\alpha)^i\right):=\min_{i\geq 0}\{\bar{v}(c_i)+i\delta\}, \, c_i\in\overline{K},$$
	is a valuation on $\overline{K}[X]$ and can be uniquely extended to $\overline{K}(X)$  (cf. \cite[Theorem 2.2.1]{En-Pr}). Such a valuation is said to be defined by $\min,\, \bar{v},\, \alpha$ and $\delta.$ If $\overline{w}=\overline{w}_{\alpha,\delta},$ then we say that $(\alpha,\delta)$ is a pair of definition for $w.$
	
	\begin{definition}
		A pair $(\alpha,\delta)$ in $\overline{K}\times \Gamma_{\overline{w}}$ is called a \emph{minimal pair} of definition for $w$  if 
	 $\overline{w}=\overline{w}_{\alpha,\delta}$ and 
	  for every $\beta$ in $\overline{K},$ satisfying $\bar{v}(\alpha-\beta)\geq\delta,$ we have  $\deg\beta\geq\deg\alpha,$ where by $\deg\alpha$ we mean the degree of the extension $K(\alpha)/K.$ 
	\end{definition}
	\begin{remark}
			In the above definition, if $\Gamma_{\overline{w}}=\Gamma_{\bar{v}},$ then  the minimal  pair $(\alpha,\delta)$ is  called a $(K,v)$-\emph{minimal pair}.
		\end{remark}
	Let $(K,v)$ be a henselian valued field. 
If  $\theta\in K,$ then $(\theta,\delta)$ is a minimal pair for each $\delta\in\Gamma_{\bar{v}}$ and it is immediate from the definition  that a pair $(\theta,\delta)\in(\overline{K}\setminus K)\times \Gamma_{\bar{v}}$ is minimal if and only if $\delta$ is strictly greater than each element of the set $M(\theta, K)$ defined by  
\begin{align*}
	M(\theta, K):=\{\bar{v}(\theta-\beta)\mid\beta\in\overline{K},~ \deg\beta<\deg\theta\}.
\end{align*}
This led to the notion of main invariant 
\begin{align*}
	\delta_K(\theta):=\max M(\theta, K)
\end{align*}
defined for those $\theta\in\overline{K}\setminus K$ for which the set $M(\theta, K)$ contains a maximum. In general, this maximum value may not exist. However, in 2002, Aghigh and Khanduja gave some necessary and sufficient condition under which the set $M(\theta, K)$ has a maximum element for every $\theta\in\overline{K}\setminus K $ (see \cite[Theorem 1.1]{AK1}).
We now recall the notion of distinguished pairs which was introduced by Popescu and Zaharescu \cite{PZ}, for local fields in 1995 and was later generalized to arbitrary henselian valued fields (\cite{AK1} and \cite{AN1}).
\begin{definition}[\bf Distinguished pairs]
	A pair $(\theta,\alpha)$ of elements of $\overline{K}$ is called a $(K,v)$-\emph{distinguished pair}  if the following conditions are satisfied:
	\begin{enumerate}[(i)]
		\item $\deg\theta>\deg\alpha,$
		\item $\bar{v}(\theta-\alpha)=\max\{\bar{v}(\theta-\beta)\mid \beta\in\overline{K},\, \deg\beta<\deg\theta \}=\delta_K(\theta),$
		\item if $\eta\in\overline{K}$ be such that $\deg\eta<\deg\alpha,$ then $\bar{v}(\theta-\eta)<\bar{v}(\theta-\alpha).$
	\end{enumerate}
\end{definition}
Equivalently, we say that $(\theta,\alpha)$ is a distinguished pair, if $\alpha$ is an element in $\overline{K}$  of minimal degree over $K$ such that $$\bar{v}(\theta-\alpha)=\delta_K(\theta).$$
Clearly (iii) implies that $(\alpha,\bar{v}(\theta-\alpha))$ is a $(K,v)$-minimal pair. Also for any two monic irreducible polynomials $f$ and $g$ over $K,$ we call $(g,f)$ a distinguished pair, if there exists a root $\theta$ of $g$ and a root $\alpha$ of $f$ such that $(\theta,\alpha)$ is a $(K,v)$-distinguished pair.
Distinguished pairs give rise to distinguished chains in a natural manner. A chain $\theta=\theta_r,\theta_{r-1},\ldots,\theta_0$ of elements of $\overline{K}$ is called a \emph{saturated distinguished chain}  for $\theta$  of length $r,$ if $(\theta_{i+1},\theta_{i})$ is a $(K,v)$-distinguished pair for every $0\leq i\leq r-1$ and $\theta_0\in K.$ 
\begin{definition}
	Let $w$ be a valuation of $K(X)$ and $\overline{w}$ a fixed common extension of $w$ and $\bar{v}$ to $\overline{K}(X).$
For any polynomial $f$ in $K[X],$ we call a root $\alpha$ of $f$ in $\overline{K}$ an \emph{optimizing root} of $f$ if 
$$\overline{w}(X-\alpha)=\max\{\overline{w}(X-\alpha')\mid f(\alpha')=0\}=\delta(f).$$ We call $\delta(f)$  the \emph{optimal value} of $f$ with respect to $\overline{w}.$
\end{definition}

	\begin{definition}[\bf Abstract key polynomials]
		A monic polynomial $Q$ in $K[X]$ is said to be an \emph{abstract key polynomial}  (abbreviated as ABKP) for $w$ if for each polynomial $f$ in $K[X]$   with $\deg f< \deg Q$ we have $\delta(f)<\delta(Q).$
	\end{definition}
	It is immediate from the definition that all monic  linear polynomials are ABKPs for $w.$  Also an ABKP for $w$ is an irreducible polynomial (see   \cite[Proposition 2.4]{NS}).
	\begin{definition}
		For a polynomial $Q$ in $K[X]$ the \emph{$Q$-truncation} of $w$ is a map $w_Q:K[X]\longrightarrow \Gamma_w$ defined by 
		$$ w_Q(f):= \min_{0\leq i\leq n}\{w(f_iQ^i)\},$$
		where $\sum_{i=0}^{n} f_i Q^i,$ $\deg f_i <\deg Q,$  is the $Q $-expansion of $f.$ 
	\end{definition}
	The $Q$-truncation  $w_Q$  of $w$ need not be a valuation \cite[Example 2.5]{NS}. However,  if $Q$ is an ABKP for $w,$ then $w_Q$ is a valuation on $K(X)$ (see \cite[Proposition 2.6]{NS}). Also any ABKP, $Q$ for $w,$ is also an ABKP for the truncation valuation $w_Q.$
For   an ABKP,  $Q$ in $K[X]$ for $w,$  we set 
\begin{align*}
	\alpha(Q):=&\min\{\deg f\mid f\in K[X],~ w_Q(f)<w(f)\}, ~
	\text{(if $w_Q=w,$ then  $\alpha(Q):=\infty$) and}\\
	\psi(Q):=&\{f\in K[X]\mid f \text{ is monic,}  w_Q(f)<w(f)\,  \text{and} \deg f=\alpha(Q) \}. 
\end{align*}
Clearly $\alpha(Q)\geq \deg Q.$ Also, 
observe that $w_Q$ is a proper truncation of $w,$  (i.e.,   $w_Q<w$)  if and only if $\psi(Q)\neq \emptyset.$  
\begin{lemma}[Lemma 2.11, \cite{NS}] \label{1.2.6}
	If $Q$ is an ABKP for $w,$ then every element $F\in\psi(Q)$ is also an ABKP for $w$ and $\delta(Q)<\delta(F).$
\end{lemma}

Let $Q$ be an ABKP for $w$ and suppose that $Q$ has a saturated distinguished chain. In the following result we prove  that each member of this chain is also an ABKP for $w,$ further, we also observe that  optimal values and main invariants associated with a saturated distinguished chain of polynomials are closely related. 
\begin{proposition}\label{1.1.6}
Let $(K,v)$ be a henselian valued field and $\bar{v}$ a unique extension of $v$ to $\overline{K}.$	Let  $Q$ be an ABKP for a valuation $w$ of $K(X).$ If $(Q_r=Q, Q_{r-1},\ldots, Q_0)$ is a saturated distinguished chain for $Q.$ Then
	\begin{enumerate}[(i)]
		\item Each $Q_i,$ $0\leq i\leq r-1,$ is an ABKP for $w.$
		\item $\delta(Q_r)>\delta(Q_{r-1})>\cdots>\delta(Q_0)$ and for optimizing roots $\theta_{i}$ of $Q_i,$ $0\leq i\leq r, $  
		$$\delta_K(\theta_{i})=\delta(Q_{i-1})~ \forall ~1\leq i\leq r.$$
	\end{enumerate} 
\end{proposition}

	The following result can be easily deduced from Theorem 1.1 of \cite{NSD} and Theorem 1.1 of \cite{JN},  and  gives a characterization of valuation-transcendental extensions.
	\begin{theorem}\label{1.1.2}
		An extension $w$ of $v$ to $K(X)$ is valuation-transcendental if and only if  $w=w_Q,$ for some ABKP, $Q$ for  $w.$ Moreover, if $\alpha$ is an optimizing root of $Q,$ then $(\alpha,\delta(Q))$ is a minimal pair of definition for $w$ and  $w=w_Q=\overline{w}_{\alpha,\delta(Q)}|_{K(X)}.$ 
			\end{theorem}
			It is known that if $w$ is a valuation-transcendental extension of $v$ to $K(X)$ defined by some minimal pair $(\theta,\delta)\in\overline{K}\times\Gamma_{\overline{w}},$ then the minimal polynomial of $\theta$ over $K$ is  an ABKP for $w$ (see \cite{NSD}). Therefore, keeping this in mind the next result follows immediately from Proposition \ref{1.1.6}.
		\begin{corollary}\label{1.1.7}
			Let $(K,v)$ be a henselian valued field and $w$ be a valuation-transcendental extension of $v$ to $K(X)$ defined by some minimal pair $(\theta,\delta)\in(\overline{K}\setminus K)\times\Gamma_{\overline{w}},$ and let $(\theta=\theta_r, \theta_{r-1},\ldots, \theta_0)$ be a saturated distinguished chain for $\theta.$  Then the  minimal polynomials $Q_i$ of $\theta_{i}$ over $K,$  $0\leq i\leq r,$ are  ABKPs for $w.$
		\end{corollary}

	\begin{definition}
		A family $\Lambda=\{Q_i\}_{i\in\Delta}$ of ABKPs for $w$ is said to be a complete set of ABKPs for $w$ if the following conditions are satisfied:
		\begin{enumerate}[(i)]
			\item $\delta(Q_i)\neq \delta(Q_j)$ for every $i\neq j\in\Delta.$
			\item $\Lambda$ is well-ordered with respect to the ordering given by $Q_i< Q_j$ if and only if  $\delta(Q_i)<\delta(Q_j)$ for every $i<j\in \Delta.$ 
			\item For any $f\in K[X],$ there exists some $Q_i \in \Lambda$ such that $w_{Q_i}(f)=w(f).$
		\end{enumerate}
	\end{definition}
	It is known that  \cite[Theorem 1.1]{NS}, every valuation $w$ on $K(X)$ admits a complete set of ABKPs. 
	Moreover, there is a complete set  $\Lambda=\{Q_i\}_{i\in\Delta}$ of ABKPs  for  $w$ having the following properties (cf.\  \cite[Remark 4.6]{MMS} and  \cite[proof of Theorem 1.1]{NS}).
	\begin{remark} \label{1.1.3}
		\begin{enumerate}[(i)]
		\item $\Delta=\bigcup_{j\in I}\Delta_j$ with $I=\{0,\ldots, N\}$ or $\mathbb{N}\cup\{0\},$ and for each $j\in I$ we have $\Delta_j=\{j\}\cup\vartheta_{j},$ where $\vartheta_j$ is an ordered set without a maximal element or is empty.
		\item $Q_0=X.$
		\item For all $j\in I\setminus \{0\}$ we have $j-1<i<j,$ for all $i \in\vartheta_{j-1}.$
		\item All  polynomials $Q_i$ with $i\in\Delta_j$ have the same degree and  have degree strictly  less than  the degree of the polynomials $Q_{i'}$ for every $i'\in\Delta_{j+1}.$
		\item For each $i<i'\in\Delta$ we have $w(Q_i)<w(Q_{i'})$ and $\delta(Q_i)<\delta(Q_{i'}).$
		\item Even though the set $\{Q_i\}_{i\in\Delta}$  of ABKPs  for $w$  is not unique, the cardinality of $I$ and the degree of an abstract key polynomial $Q_i$ for each $i\in I$ are uniquely determined by $w.$
		 \item The ordered set $\Delta$ has a maximal element if and only if the following holds:
		\begin{enumerate}[(a)]
			\item the set $I=\{0,\ldots,N\}$ is finite;
			\item $\Delta_N=\{N\},$ i.e., $\vartheta_N=\emptyset.$
		\end{enumerate}
		\end{enumerate}
	\end{remark}
From now on, we assume that all complete set of ABKPs in this paper satisfy the properties of Remark \ref{1.1.3}.

	Keeping in mind the above notations for a complete set  $\Lambda=\{Q_i\}_{i\in\Delta}$ of ABKPs for $w$  we have:
	\begin{definition}[\bf Limit key polynomials]\label{1.1.22}
		For	an element $i\in\Delta,$ we say that  $Q_i$ is a \emph{limit key polynomial} if the following conditions hold:
		\begin{enumerate}[(i)]
			\item $i\in I\setminus\{0\}.$
			\item $\vartheta_{i-1}\neq \emptyset.$
		\end{enumerate}
	\end{definition}

 In Theorem 1.23 of \cite{S-A}, it is proved that if $\{Q_i\}_{i\in\Delta}$ is a complete set of ABKPs for $w$ with $\vartheta_i=\emptyset,$ for every $i\in I,$ then there  exist some $n\in I\setminus\{0\}$ such that $Q_n$ has a saturated distinguished chain. In the next result, we show that the converse of this result also holds for a  valuation-transcendental extension.
\begin{theorem}\label{1.1.10}
	Let $(K,v)$ be a henselian valued field and $w=w_Q$  a valuation-transcendental extension of $v$ to $K(X).$ If $(Q_r=Q, Q_{r-1},\ldots, Q_0)$ is a saturated distinguished chain for $Q,$ then $\Lambda=\{Q_0\}\cup \{Q_1\}\cup\cdots\cup \{ Q_r\}$ is a complete set of ABKPs for $w.$
\end{theorem}
\bigskip
The notion of Okutsu frame was introduced by Okutsu in 1982 for local fields \cite{Oku}, and then studied by Nart in \cite{GME},  which was later generalized to henselian valued field of  arbitrary rank in \cite{Ol-Na} and  \cite{Ol}.

To define Okutsu frames, we first recall  some notations. Let $(K,v)$ be a henselian valued field of arbitrary rank.
Let $F$ in $K[X]$ be a monic irreducible polynomial of degree $n$ and $\theta\in\overline{K}$ be a root of $F.$
Consider the sequences:
\begin{align*}
	m_0=1<m_1<\cdots<m_r&=n \\
	\mu_0<\mu_1<\cdots<\mu_r&=\infty,
\end{align*}
defined in the following recurrent way:
\begin{align*}
	\mu_i:&= \max\{\bar{v}(\theta-\eta)\mid \eta\in\overline{K},~ \deg\eta=m_i\}~\text{for every $0\leq i\leq r-1,$}\\
	m_{i}&:=\min\{\deg\eta\mid \eta\in\overline{K}, ~\bar{v}(\theta-\eta)>\mu_{i-1}\}~\text{for every $1\leq i\leq r-1.$}
\end{align*}
Since $(K,v)$ is henselian, so these values does not  depend upon the choice of the root $\theta$ of $F.$ 
\begin{definition}[\bf Okutsu frames]
	For a monic irreducible polynomial $F\in K[X]$ having a root $\theta\in\overline{K},$  let $\theta_i\in\overline{K}$  be  such that $\deg\theta_i=m_i,$ $\bar{v}(\theta-\theta_i)=\mu_i,$ for every $0\leq i\leq r-1.$ If $F_i$ is the minimal polynomial of $\theta_i$ over $K,$ then the chain $[F_0, F_1,\ldots,F_{r-1}] $ of monic irreducible polynomials is called an \emph{Okutsu frame} of $F.$
\end{definition}
\begin{remark}
	It can be  shown that the above definition of an Okutsu frame  is equivalent to the one given in \cite{Ol-Na}.
\end{remark}
In the next result we give a connection between saturated distinguished chains and Okutsu frames. It may be pointed that a similar result is also proved in Theorem 2.6 of \cite{Ol} and Corollary 3.5 of \cite{JS} but our proof is elementary.
\begin{theorem}\label{1.1.9}
	Let $(K,v)$ be a henselian valued field and $F$ in $K[X]$ be a monic irreducible polynomial having a root $\theta$ in $\overline{K}.$  Then    $(F=F_{r}, F_{r-1},\ldots, F_0)$ is a saturated distinguished chain for $F$ if and only if $[F_0,F_1,\ldots,F_{r-1}]$ is an Okutsu frame of $F.$
\end{theorem}

\bigskip
We now recall  the  notion of key polynomials which was first introduced by Maclane  in 1936 and  later generalized by Vaqui\'e in 2007 (see \cite{M} and \cite{V}).  
\begin{definition}[\bf Key polynomials]
	For a valuation $w$ on $K(X)$ and polynomials $f,$ $g$ in $K[X],$ we say that
	\begin{enumerate}[(i)]
		\item  $f$ and $g$ are $w$-equivalent and write $f\thicksim_{w} g$ if $w(f-g)>w(f)=w(g).$
		\item $g$ is $w$-divisible by $f$ or  $f$  $w$-divides $g$ (denoted by $f\mid_{w}g$) if there exist some polynomial $h \in K[X]$ such that $g\thicksim_{w} fh.$
		\item  $f$ is $w$-irreducible, if for any $h,\, q\in K[X],$ whenever $f\mid_{w} hq,$ then either $f\mid_{w}h $ or $f\mid_{w}q.$
		\item $f$ is $w$-minimal, if for every polynomial $h\in K[X],$ whenever $f\mid_{w}h,$ then $\deg h\geq \deg f.$
		\item Any monic polynomial $f$  satisfying (iii) and (iv) is called a  \emph{key polynomial} for $w.$
	\end{enumerate}
	\end{definition}
In view of  Proposition 2.10 of  \cite{Ma},  any ABKP, $Q$ for $w$ is a key polynomial for $w_Q$ of minimal degree.
Let $KP(w)$ denote the set of all key polynomials for valuation $w.$ For any $\phi\in KP(w)$ we denote by $[\phi]_w$ the set of all key polynomials which are $w$-equivalent to $\phi.$
\begin{definition}[\bf Ordinary augmentation]\label{1.1.15}
	Let $\phi$ be a key polynomial for a valuation $w$ and $\gamma> w(\phi)$ be an element of a totally ordered abelian group $\Gamma$ containing $\Gamma_w$ as an ordered subgroup. A map $w': K[X]\longrightarrow \Gamma\cup\{\infty\}$ defined by 
	$$w'(f)=\min\{w(f_i)+i\gamma\},$$ where
	 $\sum_{i\geq 0}f_i \phi^i, $ $\deg f_i<\deg \phi,$ is the $\phi$-expansion of $f\in K[X],$  gives a valuation on $K(X)$ (see \cite[Theorem 4.2]{M}) called the   \emph{ordinary augmentation of $w$} (or an \emph{augmented valuation}) and is  denoted  by $w'=[w; \phi,\gamma].$
\end{definition}
 Clearly, $w(\phi)<w'(\phi)$ and
the polynomial $\phi$ is a key polynomial of minimal degree for the augmented valuation $w'$ (see \cite[Corollary 7.3]{EN}). 
\begin{definition}
	An extension $w$ of $v$ to $K(X)$ is called \emph{commensurable} if $\frac{\Gamma_{w}}{\Gamma_v}$ is a torsion group; otherwise it is called \emph{incommensurable}.
\end{definition}
Note that, if $\frac{\Gamma_{w}}{\Gamma_v}$ is a torsion group, then we have a canonical embedding $$\Gamma_{w}\hookrightarrow\Gamma_{v}\otimes\mathbb{Q}.$$

It is known that  any incommensurable extension $w$ of $v,$  is  value-transcendental (cf.\ \cite[Theorem 4.2]{EN}). Moreover,  if $\phi\in K[X]$ is a monic polynomial of minimal degree such that $w(\phi)\notin \Gamma_{v}\otimes\mathbb{Q},$ then $\phi$ is a key polynomial for $w$ and  the set of all key polynomials for $w$ is given by  
 $$\{\phi+g\mid g\in K[X], ~\deg g<\deg \phi,~ w(g)>w(\phi)\}=[\phi]_w.$$  In particular, every key polynomial for a value-transcendental extension have the same degree.
 On the other hand if $w$ is any commensurable extension of $v$ to $K(X),$ then $w$ is either valuation-algebraic or is residually transcendental. In fact
  any commensurable extension  which admits key polynomials are always residually transcendental  (see \cite[Theorem 4.6]{PP}). Hence the set of all key polynomials for a  valuation-algebraic extension $w$ is an empty set, i.e., $KP(w)=\emptyset.$ 
 
 Let  $w$ be a valuation on  $K(X),$  with value group $\Gamma_{w},$  which admits key polynomials.
 If $\phi$ is a key polynomial for $w$ of minimal degree, then we define
   $$\deg(w):=\deg \phi.$$  For any valuation $w'$ on $K(X)$ taking values in a subgroup of $\Gamma_{w},$ we say that $$w'\leq w~\text{ if and only if}~ w'(f)\leq w(f)~\forall~f\in K[X].$$   Suppose that $w'<w$ and consider the set
 $$\overline{\Phi}_{w',w}:=\{f\in K[X]\mid w'(f)<w(f)\}.$$ If $d$ is the smallest degree of a polynomial in $\overline{\Phi}_{w',w},$ then we define 
 $$\Phi_{w',w}:=\{g\in\overline{\Phi}_{w',w}\mid g~\text{is monic and  $\deg g=d$} \},$$
 i.e.,  the set of all monic polynomials $g\in K[X]$ of minimal degree such that $w'(g)<w(g).$  
 
 \begin{theorem}[Theorem 1.15, \cite{V}]\label{1.1.19}
 	Let $w$ be a valuation on $K(X)$ and $w'<w.$ Then 
  any $\phi\in \Phi_{w',w}$ is a key polynomial for $w'$ and  $$w'<[w';\phi,w(\phi)]\leq w.$$  For any non-zero polynomial $f\in K[X],$ the equality 
 $w'(f)=w(f)$ holds if and only if $\phi\nmid_{w'} f.$
\end{theorem}
 \begin{corollary}[Corollary 2.5, \cite{EN1}]\label{1.1.16}
 	Let $w'<w$ be as above. Then
 	\begin{enumerate}[(i)]
 		\item  $\Phi_{w',w }=[\phi]_{w'}$ for all $\phi\in\Phi_{w',w }.$
 		\item If $w'<\mu<w$ is a chain  of valuations, then $ \Phi_{w',w }=\Phi_{w',\mu}.$
 		In particular, 
 		\begin{align*}
 			w'(f)=w(f) \iff w'(f)=\mu(f),\hspace{2pt}\forall f\in K[X]. 
 		\end{align*}
 	\end{enumerate}
 \end{corollary}
Keeping in mind the above results, 
   we can now define 
  $$\deg(\Phi_{w',w }):=\deg(\phi) \hspace{2mm}\forall \phi\in\Phi_{w',w }.$$
 \begin{remark}\label{1.1.17}
 	Let $w$ be a valuation of $K(X).$
 	 If $w'=w_{Q'},$ for some ABKP, $Q'$ for $w,$ then $\Phi_{w',w}=\psi(Q').$ 
 \end{remark}

Consider the group $\mathbb{Z}\times(\Gamma_{v}\otimes\mathbb{Q})$ equipped with the lexicographical ordering containing $\mathbb{Z}\times\Gamma_{v}$ as an ordered subgroup. Let $w_{-\infty}: K[X]\longrightarrow(\mathbb{Z}\times\Gamma_{v})\cup\{\infty\}$ be the valuation defined by 
$$w_{-\infty}(f)=(-\deg f, v(a_n)),$$ where $a_n$ is the leading coefficient of the polynomial $f\in K[X].$    Since the value group, $\mathbb{Z}\times\Gamma_{v}$ is torsion free over $\Gamma_{v},$ so the extension $w_{-\infty}$ of  $v$ is incommensurable and in view of Lemma 4.1 of \cite{EN}, the set of all key polynomials for $w_{-\infty}$ is 
$$KP(w_{-\infty})=\{X+a\mid a\in K\}=[X]_{w_{-\infty}}.$$
Fix an  order-preserving embedding $\Gamma_v\otimes\mathbb{Q}\hookrightarrow \mathbb{Z}\times(\Gamma_{v}\otimes\mathbb{Q})$ of ordered  abelian groups, mapping $\gamma\mapsto  (0,\gamma).$
 If $w$ is any commensurable extension of $v$ to $K(X),$ then from the above embedding we have $w_{-\infty}<w,$ and $w_{-\infty}$ is called the \emph{minimal extension} of $v$ to $K(X).$  Moreover  the augmentation of $w_{-\infty}$ with respect to the key polynomial $\phi_0=X+a$ for some $a\in K,$ defined in a natural way, is denoted by  $w_0$  (see \cite[Subsection 2.2]{Ol-Na}).

\begin{definition}
A valuation $w$ is said to be \emph{inductive} if it is attained after a finite number of augmentation steps starting with the minimal valuation:
\begin{align}\label{1.1}
	w_{-\infty}\xrightarrow{\phi_0,\gamma_0}w_0\xrightarrow{\phi_1,\gamma_1} w_1\xrightarrow{\phi_2,\gamma_2}\cdots \longrightarrow w_{r-1}\xrightarrow{\phi_r,\gamma_r} w_r=w,
\end{align}
where $\gamma_0,\gamma_1,\ldots,\gamma_r\in \Gamma_{v}\otimes\mathbb{Q}$ and   $w_i=[w_{i-1}, \phi_i, \gamma_i],$ $1\leq i\leq r.$ 
\end{definition}
 The minimal valuation $w_{-\infty}$ is  not considered  an inductive valuation and  thus, inductive valuations are commensurable  which admits key polynomials.  In view of Corollary 7.3 of \cite{EN}, $\phi_i$ is a key polynomial for $w_i$ of minimal degree. 
 \begin{definition}[\bf Optimal Maclane chain]
 	A  chain of augmentations of the form  (\ref{1.1}) such that, 
 	$$1=m_0\mid m_1\mid\cdots\mid m_r,~ m_0<m_1<\cdots<m_r,$$ where  $m_i=\deg \phi_i,$  $0\leq i\leq r,$ is called the \emph{optimal Maclane chain} of $w.$
 \end{definition}
 It is known that all inductive valuations admit an optimal Maclane chain \cite{Ol-Na}. These chains are not unique, but 
\begin{itemize}
	\item the intermediate valuations $w_0,~w_1,\ldots,w_{r-1},$
	\item the degrees $m_0,~ m_1, \ldots, m_r$ of the key polynomials,
	\item  $\gamma_0,\gamma_1,\ldots,\gamma_r$ satisfies $\gamma_i=w_i(\phi_i)=w(\phi_i)$ for all $0\leq i\leq r,$ 
\item 	$\lambda_0=\gamma_0=w(\phi_0), ~ 	\lambda_i=w_i (\phi_i)-w_{i-1}(\phi_i)>0,~ 1\leq i\leq r.$
	\end{itemize}
 are some intrinsic invariants of $w.$

We now give a relation between optimal Maclane chains and complete set of ABKPs  for a valuation $w.$ We first give a precise complete set $\{Q_i\}_{i\in\Delta}$ of ABKPs using a given optimal Maclane chain. 
\begin{theorem}\label{1.1.8}
	Let $(K,v)$ be a valued field and $w$ be an extension of $v$ to $K(X).$  If $w$ has an optimal Maclane chain of the form (\ref{1.1}), then  $\{\phi_0\}\cup\{\phi_1\}\cup\cdots\cup\{\phi_r\}$ is a complete set of ABKPs for $w.$    
\end{theorem}
In the next result we gives some necessary conditions under which an optimal Maclane chain is obtained using a complete set of ABKPs for $w.$ 
\begin{theorem}\label{1.1.21}
	Let $(K,v)$ and  $(K(X),w)$ be as in the above theorem. Let $\{Q_i\}_{i\in\Delta}$ be a complete set of ABKPs for $w$  such that $\Delta$ has maximal element, say, $N$ and  $\vartheta_{i}=\emptyset$ for every $0\leq i\leq N.$    Then $w$ has an optimal Maclane chain,  if $w(Q_N)\in\Gamma_v\otimes\mathbb{Q}.$ 
\end{theorem}


\bigskip

Let $(K,v)$ be a valued field and $w$ be an extension of $v$ to $K(X).$ We now recall the definition of a  continuous family of augmentations of $w$  (\cite{V, EN}).
\begin{definition}\label{1.1.14}
	Let $w$ be a valuation on $K(X)$ admitting key polynomials. Then a continuous family of augmentations of $w$ is a family of ordinary augmentations of $w$  $$\mathcal{W} =(\rho_i=[w;\chi_i,\gamma_i])_{i\in	\mathbf{A}},$$ indexed by a set $\mathbf{A},$ satisfying the following conditions:
	\begin{enumerate}[(i)]
		\item The set $\mathbf{A}$ is totally ordered  and has no maximal element.
		\item All  key polynomials $\chi_i\in KP(w)$ have the same degree.
		\item For all $i<j$ in $\mathbf{A},$ $\chi_j$ is a  key polynomial for $\rho_i$ and satisfies:
		$$\chi_j\not\thicksim_{\rho_i}\chi_i~\text{and}~ \rho_j=[\rho_i;\chi_j,\gamma_j].$$
		The common degree $m=\deg\chi_i,$ for all $i,$ is called the \emph{stable degree} of the family $\mathcal{W}$ and is denoted by $\deg (\mathcal{W}).$
	\end{enumerate}
\end{definition}
A polynomial $f$ in $K[X]$ is said to be \emph{stable} with respect to the family $\mathcal{W}=(\rho_i)_{i\in\mathbf{A}}$ (or is $\mathcal{W}$-stable) if 
$$\rho_i(f)=\rho_{i_0}(f),~ \text{for every}~ i\geq i_0$$ for some index $i_0\in\mathbf{A}.$ This stable value is denoted by $\rho_\mathcal{W}(f).$ By Corollary \ref{1.1.16} (ii), a polynomial $f\in K[X]$  is $\mathcal{W}$-unstable if and only if  
$$\rho_i(f)<\rho_j(f)\hspace{5pt} \forall i<j.$$  The minimal degree of an $\mathcal{W}$-unstable polynomial is denoted by $m_{\infty}.$ If  all polynomials are $\mathcal{W}$-stable, then we set $m_{\infty}=\infty$. 
\begin{remark}\label{1.1.18}
The following properties hold for any  continuous family $\mathcal{W} =(\rho_i)_{i\in\mathbf{A}}$ of augmentations (see  \cite[p.\ 9]{EN1}):
\begin{enumerate}[(i)]
	\item The mapping defined by $i\rightarrow\gamma_i$ and $i\rightarrow\rho_i$ are isomorphisms of ordered sets between $\mathbf{A}$ and $\{\gamma_i\mid i\in\mathbf{A}\},$ $\{\rho_i\mid i\in\mathbf{A}\},$ respectively.
	\item For all $i\in\mathbf{A},$ $\chi_i$ is a key polynomial for $\rho_i$ of minimal degree.
	\item For all $i, j\in\mathbf{A},$ $\rho_i(\chi_j)=\min\{\gamma_i,\gamma_j\}.$ Hence, all the polynomials $\chi_i$ are stable.
	\item $\Phi_{\rho_i,\rho_j}=[\chi_j]_{\rho_i}$ $\forall$ $i<j\in\mathbf{A}.$
	\item All valuations $\rho_i$ are residually transcendental.
	\item All the value groups $\Gamma_{\rho_i}$ coincide and   the common value group is denoted by $\Gamma_{\mathcal{W}}.$ 
\end{enumerate}
\end{remark}
\begin{definition}[\bf Maclane-Vaqui\'e limit key polynomials]
	Let $\mathcal{W}$ be a continuous family of augmentations of a valuation $w.$ 
 A   monic $\mathcal{W}$-unstable polynomial of minimal degree is called \emph{Maclane-Vaqui\'e  limit key polynomial} (abbreviated as MLV) for $\mathcal{W}.$
\end{definition}
  We denote by $KP_{\infty}(\mathcal{W})$  the set of all MLV limit key polynomials. Since the product of stable polynomials is stable, so all MLV limit key polynomials are irreducible in $K[X].$ 
 \begin{definition}
 	We say that $\mathcal{W}$ is an \emph{essential continuous family} of augmentations if $m<m_{\infty}<\infty.$
 \end{definition}
\begin{remark}
	All essential continuous family of augmentations admit MLV limit key polynomials.
\end{remark}
Let $\mathcal{W}$ be an essential continuous family of augmentations of a valuation $w$
and $Q\in KP_{\infty}(\mathcal{W})$ be any MLV limit key polynomial. Then any polynomial $f$ in $K[X]$ with $\deg f<\deg Q$ is  $\mathcal{W}$-stable.
\begin{definition}[\bf Limit augmentation]
	Let   $Q$ be any MLV limit key polynomial for an essential continuous family of augmentations $\mathcal{W} =(\rho_i)_{i\in\mathbf{A}}$  and   $\gamma>\rho_i(Q),$ for all $i\in\mathbf{A},$ be an element of  a totally ordered abelian group $\Gamma\cup\{\infty\}$  containing $\Gamma_{\mathcal{W}}$ as an  ordered subgroup. Then a map $w': K[X]\longrightarrow\Gamma\cup\{\infty\}$ defined by
	$$w'(f)=\min_{i\geq 0}\{\rho_\mathcal{W}(f_i)+i\gamma\},$$
where  $\sum_{i\geq 0} f_iQ^i,$ $\deg f_i<\deg Q,$ is the $Q$-expansion of $f\in K[X],$ gives a valuation on $K(X)$ and is called  the \emph{limit augmentation} of $\mathcal{W}.$ 
\end{definition}
 Note that $w'(Q)=\gamma$ and $\rho_i<w'$ for all $i\in\mathbf{A}.$ If $\gamma<\infty,$ then $Q$ is a key polynomial for $w'$ of minimal degree \cite[Corollary 7.13]{EN}).
\vspace{.20pt}

We now recall the definition of Maclane-Vaqui\'e chains given by  Nart in \cite{EN1}. For this,  we first
consider a finite, or countably infinite, chain of mixed augmentations
\begin{align}\label{1.12}
	w_0\xrightarrow{\phi_1,\gamma_1} w_1\xrightarrow{\phi_2,\gamma_2}\cdots \longrightarrow w_{n}\xrightarrow{\phi_{n+1},\gamma_{n+1}} w_{n+1}\longrightarrow\cdots
\end{align}
in which every valuation is an augmentation of the previous one and is of one of the following type:
\begin{itemize}
	\item  Ordinary augmentation: $w_{n+1}=[w_n; \phi_{n+1},\gamma_{n+1}],$ for some $\phi_{n+1}\in KP(w_n).$ 
	\item Limit augmentation:  $w_{n+1}=[\mathcal{W}_n; \phi_{n+1},\gamma_{n+1}],$ for some $\phi_{n+1}\in KP_{\infty}(\mathcal{W}_n),$ where $\mathcal{W}_n$ is an essential continuous family of augmentations of $w_n.$ 
\end{itemize}
Let $\phi_0\in KP(w_0)$ be a key polynomial of minimal degree and let $\gamma_0=w_0(\phi_0).$
 Then, in view of  Theorem \ref{1.1.19},  Proposition 6.3 of \cite{EN}, Proposition 2.1, 3.5 of \cite{EN1} and Corollary \ref{1.1.16},  we have the   following properties of a  chain (\ref{1.12}) of augmentations.
 \begin{remark}\label{1.1.20}
\begin{enumerate}[(i)]
	\item  $\gamma_n=w_n(\phi_n)<\gamma_{n+1}.$
	\item For all $n\geq 0$ for which $\gamma_n<\infty,$ the polynomial $\phi_n$ is a key polynomial for $w_n$ of minimal degree and therefore
	$$\deg(w_n)=\deg \phi_n~\text{divides}~\deg (\Phi_{w_n,w_{n+1}}).$$
	\item \begin{equation*}
		\Phi_{w_n,w_{n+1}}=
		\begin{cases}
			[\phi_{n+1}]_{w_n},~ \text{if}~ w_n\rightarrow w_{n+1}~ \text{is ordinary augmentation}\\
		\displaystyle\Phi_{w_n,\mathcal{W}_n}=[\chi_i]_{w_n},~\text{ if}~ w_n\rightarrow w_{n+1} ~\text{is limit augmentation}.
		\end{cases}
	\end{equation*}
	\item \begin{equation*}
			\deg(\Phi_{w_n,w_{n+1}})=
			\begin{cases}
				\deg\phi_{n+1},~\text{if}~ w_n\rightarrow w_{n+1}~ \text{is ordinary augmentation}\\
				\displaystyle\deg(\mathcal{W}_n),~\text{ if}~ w_n\rightarrow w_{n+1} ~\text{is limit augmentation}.
			\end{cases}
	\end{equation*}
\end{enumerate}
\end{remark}
Let $ a \in K,$ $\gamma\in\Gamma\cup\{\infty\}.$ Then the valuation defined by the pair $(a,\gamma)$
  is called a \emph{depth zero valuation}.
\begin{definition}[\bf Maclane-Vaqui\'e chains]\label{1.1.13}
	A finite, or countably infinite chain of mixed augmentations as in (\ref{1.12}) is called a \emph{Maclane-Vaqui\'e  chain} (abbreviated as MLV), if every augmentation step satisfies:
	\begin{itemize}
		\item if $ w_n\rightarrow w_{n+1}$ is ordinary augmentation, then $\deg (w_n)<\deg(\Phi_{w_n,w_{n+1}}).$
		\item if $ w_n\rightarrow w_{n+1}$ is limit augmentation, then $\deg (w_n)=\deg(\Phi_{w_n,w_{n+1}})$ and $\phi_n\notin\Phi_{w_n,w_{n+1}}.$
	\end{itemize}
	A  Maclane-Vaqui\'e  chain is said to be \emph{complete} if  $w_0$ is a  depth zero valuation.
\end{definition}
In the following result, using  complete set of ABKPs for a valuation $w,$  we give a construction of a complete finite MLV chain whose last valuation is $w.$
\begin{theorem}\label{1.1.11}
	Let $(K,v)$ be a valued field and let $w$ be an extension of $v$ to $K(X).$ If $\{Q_i\}_{i\in\Delta}$ is a complete set of ABKPs for $w$  such that $N$ is the last element of $\Delta,$ then 
		\begin{align}\label{1.13}
		w_0\xrightarrow{Q_1,\gamma_1} w_1\xrightarrow{Q_2,\gamma_2}\cdots \longrightarrow w_{N-1}\xrightarrow{Q_N,\gamma_N} w_N=w,
	\end{align}
	 is a complete finite MLV chain of $w$ such that 
	\begin{enumerate}[(i)]
	 	\item  if $\vartheta_{j}=\emptyset,$ then $w_j\longrightarrow w_{j+1}$ is an ordinary augmentation. Further, $w_{j+1}=w_{Q_{j+1}},$ and $\gamma_{j+1}=w(Q_{j+1}).$
	 	\item if $\vartheta_{j}\neq \emptyset,$ then $w_j\longrightarrow w_{j+1}$ is a limit augmentation. Further, $w_{j+1}=w_{Q_{j+1}}$ and $\gamma_{j+1}=w(Q_{j+1}).$
	\end{enumerate}
\end{theorem}
The converse of the above result also holds.
\begin{theorem}\label{1.1.12}
	Let $(K,v)$ be a valued field and let $w$ be an extension of $v$ to $K(X).$ If 
	\begin{align}\label{1.14}
		w_0\xrightarrow{\phi_1,\gamma_1} w_1\xrightarrow{\phi_2,\gamma_2}\cdots \longrightarrow w_{N-1}\xrightarrow{\phi_N,\gamma_N} w_N=w,
	\end{align}
is a complete finite MLV chain, 
	 then $\{\phi_i\}_{i\in\Delta}$ forms a complete set of ABKPs for $w$ such that
	 \begin{enumerate}[(i)]
	 	\item $N$ is the last element of $\Delta.$
	 	\item if $w_j\longrightarrow w_{j+1}$ is an ordinary augmentation, then $\phi_{j+1}\in\psi(\phi_j).$
	 	\item if $w_j\longrightarrow w_{j+1}$ is a limit augmentation, then $\phi_{j+1}$ is a limit key polynomial.
	 \end{enumerate} 
\end{theorem}
Note that if  $w_j\longrightarrow w_{j+1}$ is an ordinary augmentation for every $0
\leq j\leq N$ and $\gamma_N\in\Gamma_{v}\otimes\mathbb{Q},$ then the above MLV chain is nothing but an optimal Maclane chain and Theorem \ref{1.1.12} is an immediate  consequence of Theorem \ref{1.1.8}. On the other hand, if $\vartheta_{j}=\emptyset$ for every $j$ and $w(Q_N)\in\Gamma_{v}\otimes\mathbb{Q},$ then  Theorem \ref{1.1.11} follows from  Theorem \ref{1.1.21}.

It is known that if  $\{Q_i\}_{i\in\Delta}$ is a complete set of ABKPs for $w,$ then $w$ is a valuation-transcendental extension of $v$ to $K(X)$ if and only if $\Delta$ has a maximal element, say, N, and then 	 $w=w_{Q_N}$ (see \cite[Theorem 5.6]{MMS}). Therefore, as an immediate consequence of Theorems \ref{1.1.11} and \ref{1.1.12}, we have the following result.
\begin{corollary}
	Let $(K,v)$ and $(K(X),w)$ be as above. Then the following are equivalent:
	\begin{enumerate}[(i)]
		\item The extension $w$ is valuation-transcendental.
		\item There exist a complete set  $\{Q_i\}_{i\in\Delta}$ of ABKPs for $w$ such that $\Delta$ has a maximal element.
		\item  The extension $w$ is the last valuation of a complete finite MLV chain.
	\end{enumerate}
\end{corollary}
\bigskip
	\section{Preliminaries}
	Let $(K,v)$ be a valued field and  $(\overline{K},\bar{v})$ be as before. Let $w$ be an extension of $v$ to $K(X)$  and $\overline{w}$ be a common extension of $w$ and $\bar{v}$ to $\overline{K}(X).$ In this section we give some preliminary results which will be used to prove the main results. 
	
	We first recall some basic properties of ABKPs for  $w$   (see Proposition 2.16 of \cite{S-A}, Proposition 3.8, Corollary  3.13 and Theorem 6.1 of  \cite{JN1}).
	\begin{proposition}\label{2.1.6}
		For  ABKPs, $Q$ and $Q'$ for $w$ the following holds:
		\begin{enumerate}[(i)]
				\item If $w_Q<w,$ then $w_Q$ is an r.\ t.\ extension.
			\item If  $\delta(Q)<\delta(Q'),$ then $w_Q(Q')<w(Q').$ 
			\item If $\deg Q = \deg Q',$ then $$w(Q)<w(Q')\iff w_Q(Q')<w(Q')\iff \delta(Q)<\delta(Q').$$
			\item  Let  $\delta(Q)<\delta(Q').$ For any polynomial $f\in K[X],$ if $w_{Q'}(f)<w(f),$ then $w_Q(f)<w_{Q'}(f).$	
			\item If $Q'\in \psi(Q),$ then $Q$ and $Q'$ are key polynomials for $w_Q.$ Moreover, $w_{Q'}=[w_Q; Q', w_{Q'}(Q')=w(Q')].$
		\end{enumerate}
	\end{proposition}
	The following result gives a comparison between key polynomials and ABKPs.
	\begin{theorem}[Theorem 2.17, \cite{Ma}]\label{2.1.2}
		Suppose that $w'<w$ and $Q$ is a key polynomial for $w'.$ Then $Q$ is an ABKP polynomial for $w$ if and only if it satisfies one of the following two conditions:
		\begin{enumerate}[(i)]
			\item $Q\in \Phi_{w',w},$
			\item $Q\notin \Phi_{w',w}$ and $\deg Q=\deg w'.$
		\end{enumerate}
		In the first case $w_Q=[w'; Q, w(Q)].$ In the second case $w_Q=w'.$
	\end{theorem}
The next result relates ABKPs with distinguished pairs.
\begin{lemma}\label{1.1.5}
	Let $(K,v)$ be a henselian valued field  and $w$ be an extension of $v$ to $K(X).$ Let $F$ be an ABKP for $w$ and $Q$ be any polynomial such that $(F,Q)$  is a distinguished pair.
	Then the following holds:
	\begin{enumerate}[(i)]
		\item If $\theta$ and $\alpha$ are optimizing roots of $F$ and $Q$ respectively, then $(\theta,\alpha)$ is a $(K,v)$-distinguished pair.
		\item The polynomial  $Q$ is an ABKP for $w.$ Moreover $w_Q<w.$ 
	\end{enumerate}
\end{lemma} 
\begin{proof}
	\noindent(i)~ The proof follows from Lemma 2.1 of \cite{S-A}.
	
	\noindent (ii)~	Since $\deg Q<\deg F$ and $F$ is an ABKP for $w,$ so $\delta(Q)<\delta(F),$ i.e., 
	$$\overline{w}(X-\alpha)=\delta(Q)<\delta(F)=\overline{w}(X-\theta),$$ where $\theta$ and $\alpha$ are optimizing roots of $F$ and $Q$ respectively,  which in view of strong triangle law implies that 
	\begin{align}\label{1.8}
		\bar{v}(\theta-\alpha)=\delta(Q)<\overline{w}(X-\theta).
	\end{align}
	Let $g$ in $K[X]$ be any polynomial with $\deg g<\deg Q.$ Then to prove that $Q$ is an ABKP for $w$ we need to show that $\delta(g)<\delta(Q).$ As $(F,Q)$ is a distinguished pair, so by (i), $(\theta,\alpha)$ is a $(K,v)$-distinguished pair. Now for an optimizing root $\beta$ of $g$ we have $\deg\beta<\deg \alpha,$ which in view of the fact that $(\theta,\alpha)$ is a $(K,v)$-distinguished pair implies that 
	\begin{align*}
		\bar{v}(\theta-\beta)<\bar{v}(\theta-\alpha)=\delta(Q).
	\end{align*}
	From (\ref{1.8}) and the above inequality, we have that 
	$$\delta(g)=\overline{w}(X-\beta)=\min\{\overline{w}(X-\theta),\bar{v}(\theta-\beta)\}=\bar{v}(\theta-\beta)<\delta(Q).$$ Hence $Q$ is an ABKP for $w.$ As $\delta(Q)<\delta(F),$ so by Proposition \ref{2.1.6} (ii),  $w_Q(F)<w(F),$ i.e., $w_Q<w.$
\end{proof}
The following result  gives some necessary and  sufficient conditions under which an ABKP for $w$ has a saturated distinguished chain. 
\begin{corollary}[Corollary 1.18, \cite{S-A}]\label{1.2.10}
	Let $w$ be an  extension of $v$ to $K(X)$ and $Q$ be an ABKP for $w.$  Then $Q$ has a saturated distinguished chain of ABKPs if and only if there exists ABKPs,  $Q_0, Q_1,\ldots, Q_r=Q$   for $w,$ such that $\deg Q_0=1,$ $\deg Q_{i-1}<\deg Q_{i}$ and $Q_{i}\in\psi(Q_{i-1})$ for each $i,$  $1\leq i\leq r.$
\end{corollary}
\begin{lemma}[Lemma 5.1, \cite{AK2}]\label{2.1.7}
	Let $(K,v)$ be henselian valued field.  If $(\theta,\theta_1)$ and $(\theta_1,\theta_2)$ are two distinguished pairs of elements of $\overline{K},$ then $$\delta_K(\theta)>\delta_K(\theta_1)=\bar{v}(\theta_1-\theta_2)=\bar{v}(\theta-\theta_2).$$
\end{lemma}

	\section{Proof of Main Results}

\begin{proof}[Proof of Proposition \ref{1.1.6}]
	\noindent (i) Since $(Q_r=Q, Q_{r-1},\ldots, Q_0)$ is a saturated distinguished chain for $Q,$ so $(Q_i, Q_{i-1})$ is a distinguished pair for each $1\leq i\leq r.$  In particular, for $i=r$ we have that  $(Q,Q_{r-1})$ is a distinguished pair and as $Q$ is an ABKP for $w,$ so by Lemma \ref{1.1.5} (ii), $Q_{r-1}$ is an ABKP for $w.$  Arguing similarly we get   that each $Q_i,$  $i\in\{0,\ldots, r-2\}$ is an ABKP for $w.$
	
	\medskip
	\noindent (ii) As $\theta_i,$ $0\leq i\leq r,$ is an  optimizing root of $Q_i,$ so by Lemma \ref{1.1.5} (i), $(\theta=\theta_r, \theta_{r-1},\ldots,\theta_0)$ is a saturated distinguished chain for $\theta$ and therefore $(\theta_{i},\theta_{i-1}),$  $(\theta_{i-1},\theta_{i-2})$ are distinguished pairs. Then by Lemma \ref{2.1.7}, 
	\begin{align}\label{1.10}
	\bar{v}(\theta_i-\theta_{i-1})=	\delta_K(\theta_{i})>\delta_K(\theta_{i-1})=\bar{v}(\theta_{i-1}-\theta_{i-2}).
	\end{align}
Since $Q$ is an ABKP for $w,$ so by $(i),$ each $Q_i, 0\leq i\leq r-1,$ is also an ABKP for $w$ and as $\deg Q_{i-1}<\deg Q_i,$  therefore
$$\overline{w}(X-\theta_{i-1})=\delta(Q_{i-1})<\delta(Q_i)=\overline{w}(X-\theta_{i}),$$ which in view of  strong triangle law implies that $$\delta(Q_{i-1})=\bar{v}(\theta_i-\theta_{i-1}).$$ The above equality  together with (\ref{1.10}) gives $$\delta_K(\theta_i)=\delta (Q_{i-1}).$$  Now $\delta(Q_r)>\delta(Q_{r-1})>\cdots>\delta(Q_0)$ follows from  the definition of an ABKP.
\end{proof}
		
	\begin{proof}[Proof of Theorem \ref{1.1.10}]
		By Proposition \ref{1.1.6} (i), each $Q_i,$ $0\leq i\leq r$ is an ABKP for $w.$ Since $\deg Q_{i-1}<\deg Q_i,$  so  $\delta(Q_{i-1}) <\delta(Q_i)$ and from Corollary \ref{1.2.10}, it follows that $Q_{i-1}\in\psi(Q_i),$ for each $1\leq i\leq r.$
		 Now let $f\in K[X]$ be any polynomial. If $\deg f<\deg Q_i,$ for some $0\leq i\leq r,$ then $w_{Q_i}(f)=w(f).$ On the other hand, if $\deg f\geq \deg Q_i,$ for every $0\leq i\leq r$, then by definition of $w,$ $w_{Q_r}(f)=w(f).$ Hence $\Lambda=\{Q_0\}\cup\{Q_1\}\cup\cdots\cup\{Q_r\}$ is a complete set of ABKPs for $w.$
	\end{proof}
	
	\begin{proof}[Proof of Theorem \ref{1.1.9}]
		Assume first that $(F=F_r, F_{r-1},\ldots, F_0)$ is a saturated distinguished chain for $F.$ Then there exists some root $\theta_i$ of $F_i$ such that $(\theta=\theta_{r},\theta_{r-1},\ldots,\theta_0)$ is a saturated distinguished chains for $\theta,$ i.e., for each, $0\leq i\leq r-1,$  $(\theta_{i+1}, \theta_{i})$ is a distinguished pair, which implies that  each $\theta_{i}$ is of minimal degree over $K$  such that
		\begin{align}\label{1.15}
			\bar{v}(\theta_{i+1}-\theta_{i})=\max\{\bar{v}(\theta_{i+1}-\eta)\mid \eta\in\overline{K},~\deg\eta<\deg\theta_{i+1}\}
		\end{align}
		and $\theta_0\in K.$ Since $(\theta,\theta_{r-1}),$ and $(\theta_{r-1},\theta_{r-2})$ are distinguished pairs, so by Lemma \ref{2.1.7}, we have that 
			$$\delta_K(\theta)=\bar{v}(\theta-\theta_{r-1})> \delta_K(\theta_{r-1})=\bar{v}(\theta_{r-1}-\theta_{r-2})=\bar{v}(\theta-\theta_{r-2}).$$ Again  on applying  Lemma \ref{2.1.7}, for  distinguished pairs    $(\theta_{r-1},\theta_{r-2}),$ and  $(\theta_{r-2},\theta_{r-3}),$   we get that 
			$$\delta_K(\theta_{r-1})=\bar{v}(\theta_{r-1}-\theta_{r-2})=\bar{v}(\theta-\theta_{r-2}) > \delta_K(\theta_{r-2})=\bar{v}(\theta_{r-2}-\theta_{r-3})$$ which in view of strong triangle law implies that $$\delta_K(\theta_{r-1})> \delta_K(\theta_{r-2})=\bar{v}(\theta-\theta_{r-3}).$$ On continuing in the similar manner, for every $1\leq i\leq r-1,$ we have that
		\begin{align}\label{1.17}
		\delta_K(\theta_{i+1})=\bar{v}(\theta_{i+1}-\theta_{i})=\bar{v}(\theta-\theta_{i})>\delta_K(\theta_{i})=\bar{v}(\theta-\theta_{i-1}).
	\end{align}
	Therefore	from the above inequality and (\ref{1.15}),  $\theta_{i}$ is of minimal degree over $K$ such that
	\begin{align}\label{1.18} 
		\bar{v}(\theta-\theta_{i})=\max\{\bar{v}(\theta_{i+1}-\eta)\mid\eta\in\overline{K},~\deg\eta<\deg\theta_{i+1}\}~\forall ~0\leq i\leq r-1.
	\end{align}
		    For any $\eta\in\overline{K}$ with $\deg\eta<\deg\theta_{i+1},$ we now claim that $\bar{v}(\theta_{i+1}-\eta)=\bar{v}(\theta-\eta).$ For $i=r-1,$ this holds trivially. Let $0\leq i\leq r-2,$ then as $\deg\eta<\deg\theta_{i+1}$ and $(\theta_{i+1},\theta_{i})$ is a distinguished pair, so by (\ref{1.17})
		\begin{align}\label{1.20}
			\bar{v}(\theta_{i+1}-\eta)\leq \bar{v}(\theta_{i+1}-\theta_{i})= \bar{v}(\theta-\theta_{i})<\bar{v}(\theta-\theta_{i+1}),
		\end{align} 
	which in view of the strong triangle law implies that $$ \bar{v}(\theta_{i+1}-\eta)=\bar{v}(\theta-\eta).$$
It now  follows from (\ref{1.18}) and the claim that
	 \begin{align*}
	 \delta_K(\theta_{i+1})=\bar{v}(\theta-\theta_{i})&= \max\{\bar{v}(\theta_{i+1}-\eta)\mid\eta\in\overline{K},~\deg\eta<\deg\theta_{i+1}\}\nonumber\\
	 &=\max\{\bar{v}(\theta-\eta)\mid\eta\in\overline{K},~\deg\eta<\deg\theta_{i+1}\}.
	\end{align*}
			The above equality  immediately implies that 
		\begin{align*}
		 \bar{v}(\theta-\theta_i)=\max\{\bar{v}(\theta-\eta)\mid\eta\in\overline{K},~\deg\eta=\deg\theta_{i}\},~\text{for all $0\leq i\leq r-1$}.
		 \end{align*}
		   In order to prove that $[F_0,F_1,\ldots, F_{r-1}]$ is an Okutsu frame for $F,$ in view of the above equality, and the fact that $\deg\theta_0=1,$ it only remains to show that for every $1\leq i\leq r-1,$
		   \begin{align}\label{1.21}
		    \deg\theta_{i}=\min\{\deg\eta\mid\eta\in\overline{K},~\bar{v}(\theta-\eta)>\bar{v}(\theta-\theta_{i-1})\}.
		    \end{align}
		     Let $\eta\in\overline{K}$ be such that $\deg\eta<\deg\theta_{i},$ then as $(\theta_{i},\theta_{i-1})$ is a distinguished pair, so $$\bar{v}(\theta_{i}-\eta)\leq\bar{v}(\theta_{i}-\theta_{i-1})=\bar{v}(\theta-\theta_{i-1})<\bar{v}(\theta-\theta_{i}),$$ which together with strong triangle law implies that $$\bar{v}(\theta-\eta)=\bar{v}(\theta_{i}-\eta)\leq\bar{v}(\theta-\theta_{i-1}).$$ Hence (\ref{1.21}) follows. 
		
		Conversely, let  $[F_0,F_1,\ldots,F_{r-1}]$ be an Okutsu frame for $F.$ Then there exist some root $\theta_i$ of $F_i$ such that 
		\begin{align}\label{1.22}
		\bar{v}(\theta-\theta_i)&=	\max\{\bar{v}(\theta-\eta)\mid \eta\in\overline{K},~ \deg\eta =\deg\theta_i\},~\text{ $0\leq i\leq r-1,$}\\
		 \deg\theta_{i}&=\min\{\deg\eta\mid \eta\in\overline{K},~ \bar{v}(\theta-\eta)>\bar{v}(\theta-\theta_{i-1})\},~\text{ $1\leq i\leq r-1,$}\nonumber
		\end{align}
	and 
	\begin{align}\label{1.23}
1=	\deg\theta_0<\cdots<	\deg\theta_{i}&<\deg\theta_{i+1}<\cdots<\deg\theta_{r}=\deg\theta\nonumber\\
	\bar{v}(\theta-\theta_0)<\cdots<\bar{v}(\theta-\theta_{i})&<\bar{v}(\theta-\theta_{i+1})<\cdots<\bar{v}(\theta-\theta_{r-1})<\infty.
	\end{align}
		   In order to  prove that $(F,F_{r-1},\ldots, F_0)$ is a saturated distinguished chain for $F,$ it is enough to show that   $(\theta=\theta_{r},\theta_{r-1},\ldots,\theta_0)$ is a saturated distinguished chain for $\theta.$  From (\ref{1.23}), on using strong triangle law we get that  
		   \begin{align}\label{1.5}
		   	\bar{v}(\theta-\theta_{i})=\bar{v}(\theta_{i+1}-\theta_{i}),~\text{  $0\leq i\leq r-1.$}
		   \end{align}
		    Now for any $\eta\in\overline{K}$ with $\deg\eta<\deg\theta_{i+1},$ we show that $$\bar{v}(\theta-\eta)=\bar{v}(\theta_{i+1}-\eta)~\text{and}~ \bar{v}(\theta_{i+1}-\theta_{i})\geq \bar{v}(\theta_{i+1}-\eta).$$ For $i=r-1,$ first equality holds trivially. If $\deg\eta=\deg\theta_{r-1},$ then by (\ref{1.22}), $\bar{v}(\theta-\eta)\leq\bar{v}(\theta-\theta_{r-1}),$ on the other hand, if $\deg\eta\neq\deg\theta_{r-1},$ then by definition of $\deg\theta_{r-1}$ and  (\ref{1.23}) we have that $$\bar{v}(\theta-\eta)\leq\bar{v}(\theta-\theta_{r-2})<\bar{v}(\theta-\theta_{r-1}).$$
		    Let $0\leq i\leq r-2.$ Since $\deg\eta<\deg\theta_{i+1},$ so by definition of $\deg\theta_{i+1}$ and  (\ref{1.23}), we get  
		    $$\bar{v}(\theta-\eta)\leq \bar{v}(\theta-\theta_{i})<\bar{v}(\theta-\theta_{i+1})$$ which in view of strong triangle law and (\ref{1.5}) implies that
		     $$\bar{v}(\theta_{i+1}-\eta)=\bar{v}(\theta-\eta)\leq\bar{v}(\theta-\theta_{i})=\bar{v}(\theta_{i+1}-\theta_{i}).$$
		   Hence
	   $$\bar{v}(\theta_{i+1}-\theta_{i})=\max\{\bar{v}(\theta_{i+1}-\eta)\mid\eta\in\overline{K},~\deg\eta<\deg\theta_{i+1}\}=\delta_K(\theta_{i+1})$$  and $\deg\theta_{i}$ is minimal with this property, because if 
	      there exist some $\beta\in\overline{K}$  with $\deg \beta<\deg\theta_i,$  then by definition of $\deg\theta_i$ and (\ref{1.23}), $$\bar{v}(\theta-\beta)\leq\bar{v}(\theta-\theta_{i-1})<\bar{v}(\theta-\theta_{i})=\bar{v}(\theta_{i+1}-\theta_{i}),$$ which  on using strong triangle law gives 
	     \begin{align*}
	     	\bar{v}(\theta_{i+1}-\beta)&=\min\{\bar{v}(\theta_{i+1}-\theta_{i}),\bar{v}(\theta_i-\theta),\bar{v}(\theta-\beta)\}\\
	     	&=\bar{v}(\theta-\beta)<\bar{v}(\theta-\theta_{i})=\bar{v}(\theta_{i+1}-\theta_{i})=\delta_K(\theta_{i+1}),
	     \end{align*}
     i.e., $\bar{v}(\theta_{i+1}-\beta)<\delta_K(\theta_{i+1}).$
	 As $\deg\theta_0=1,$ so $\theta_0\in K.$  Hence $(\theta=\theta_r,\theta_{r-1},\ldots,\theta_0)$ is a saturated distinguished chain for $\theta.$
	\end{proof}
\begin{remark}
	From  proof of the above theorem, we can conclude that  $[F_0,F_1,\ldots,F_{r-1}] $ is an Okutsu frame for a monic irreducible polynomial $F\in K[X],$  if there exist some root $\theta_i$ of $F_i$  such that 
	$$\bar{v}(\theta-\theta_i)=\max\{\bar{v}(\theta-\eta)\mid \eta\in\overline{K},~ \deg\eta<\deg \theta_{i+1}\} ~\text{for all $0\leq i\leq r-1$},$$ $\deg\theta_{i}$ is   minimal with this property and $\deg\theta_0=1.$
\end{remark}
	\begin{proof}[Proof of Theorem \ref{1.1.8}]
	Let	$$w_{-\infty}\xrightarrow{\phi_0,\gamma_0}w_0\xrightarrow{\phi_1,\gamma_1} w_1\xrightarrow{\phi_2,\gamma_2}\cdots \longrightarrow w_{r-1}\xrightarrow{\phi_r,\gamma_r} w_r=w,$$ be an optimal Maclane chain of $w.$ If $r=0,$ then $w_{-\infty}\xrightarrow{\phi_0,\gamma_0}w_0=w$ is an optimal Maclane chain of $w.$  Since $w_0=[w_{-\infty}; \phi_0, \gamma_0]$ and $\gamma_0=w_0(\phi_0)=w(\phi_0),$ so for any polynomial $f\in K[X],$ with $\phi_0$-expansion $\sum_{i\geq 0} a_i \phi_0 ^{i},$  we have $$w_{\phi_0}(f)=\min\{w(a_i\phi_0 ^i)\}=\min\{v(a_i)+iw(\phi_0)\}=w_0(f)=w(f).$$  Therefore, $\{\phi_0\}$ is a complete set of ABKP for $w_0=w.$ Assume now  that $r\geq 1.$ Since  $\phi_i$ is a key polynomial of minimal degree for $w_i,$ i.e., $\deg \phi_i=\deg w_i,$ and 
	  $w_i(\phi_{i})=w(\phi_i),$  so $\phi_i\notin\Phi_{w_i,w}$ which in view of Theorem \ref{2.1.2} implies that  $\phi_i,$  is an ABKP for $w.$ Moreover, 
	\begin{align}\label{1.16}
	 w_i=w_{\phi_i},~\text{$0\leq i\leq r$}.
	 \end{align}
 As $\deg \phi_{i-1}<\deg \phi_i,$ and $\phi_i$ is an ABKP for $w,$ so $$\delta(\phi_{i-1})<\delta(\phi_{i}).$$
	   Since $w_{\phi_{i-1}}<w,$  $\phi_i$ is a key polynomial for $w_{\phi_{i-1}}$ which is also an ABKP for $w,$ and $\deg(w_{\phi_{i-1}})=\deg \phi_{i-1}<\deg \phi_i,$   so   by Theorem \ref{2.1.2}, we have that $\phi_i\in\Phi_{w_{\phi_{i-1}},w}.$ 
	   Therefore, by Remark \ref{1.1.17} we get that 
	   \begin{align}\label{1.4}
	   	\phi_i\in\psi(\phi_{i-1}),~1\leq i\leq r.
	   \end{align}
	     Let $f$ in $K[X]$ be any polynomial. If $\deg f<\deg\phi_i$ for some $i,$ then $w_{\phi_i}(f)=w(f).$  Otherwise, if $\deg f\geq \deg\phi_i$ for all $0\leq i\leq r,$ then from (\ref{1.16}) and the fact that  $w_r=w,$ we have  $w_{\phi_r}=w.$ Hence $w_{\phi_r}(f)=w(f).$ Thus $\Lambda=\{\phi_0\}\cup\{\phi_1\}\cup\cdots\cup\{\phi_r\}$ is a complete set of ABKPs for $w$ such that $\vartheta_{i}=\emptyset,$ for every $0\leq i\leq r$ (by \ref{1.4}) and  $r$ is the maximal element of $\Delta.$
	\end{proof}
\begin{proof}[Proof of Theorem \ref{1.1.21}]
	Suppose   $\Lambda=\{Q_i\}_{i\in\Delta}$ is a complete set of ABKPs for $w$ such that $N$ is the maximal element of $\Delta.$ If $N=0,$  then $w$ is a depth zero valuation $w_{Q_0}=w$ and by the hypothesis that $w(Q_0)\in\Gamma_v\otimes\mathbb{Q},$ the extension $w$  is commensurable. Since $Q_0$ is a monic polynomial of degree one, so $Q_0$ is a key polynomial for $w_{-\infty},$ and by definition of $w_{-\infty}$ we have that $w_{-\infty}(Q_0)< w_{Q_0}(Q_0)=w(Q_0).$ Therefore
	$w_{Q_0}=[w_{-\infty}, Q_0,  w(Q_0)]$ is the augmentation of $w_{-\infty}$  and the result holds in this case.
	Assume now that $N\geq 1.$  Since each $\vartheta_{i}=\emptyset,$   so $Q_i\in\psi(Q_{i-1})$ for every $1\leq i\leq N,$ which in view of   
  Proposition \ref{2.1.6} (v) implies that   $Q_{i-1},$ $Q_i$    are key polynomials for $w_{Q_{i-1}}$ and  $$w_{Q_i}=[w_{Q_{i-1}};Q_i, w_{Q_i}(Q_i)=w(Q_i)]$$ is the augmentation of $w_{Q_{i-1}}.$
	   Now from Proposition \ref{2.1.6} (i), and the assumption that $w(Q_N)\in\Gamma_{v}\otimes\mathbb{Q},$ we get that  $$\gamma_i=w_{Q_i}(Q_i)=w(Q_i)\in\Gamma_{v}\otimes\mathbb{Q}, ~0\leq i\leq N.$$
	   By Remark \ref{1.1.3} (iv), $\deg Q_{i-1}<\deg Q_i$ for every $1\leq i\leq N$  and as $Q_i\in\psi(Q_{i-1}),$ so by \cite[Theorem 1.12 (ii)]{S-A}, we have that  $\deg Q_{i-1}\mid \deg Q_i.$   Arguing similarly as in the case $N=0,$ we have that $w_{Q_0}=[w_{-\infty}, Q_0,  w(Q_0)]$ is the augmentation of $w_{-\infty}.$ 
	  Hence
	$$w_{-\infty}\xrightarrow{Q_0,\gamma_0}w_{Q_0}\xrightarrow{Q_1,\gamma_1} w_{Q_1}\xrightarrow{Q_2,\gamma_2}\cdots \longrightarrow w_{Q_{N-1}}\xrightarrow{Q_N,\gamma_N} w_{Q_N}=w$$  is an optimal Maclane chain of $w.$
	\end{proof}

	\begin{proof}[Proof of Theorem \ref{1.1.11}]
		Let $\{Q_i\}_{i\in\Delta}$ be a complete set of ABKPs for $w$ with $N$ the maximal element of $\Delta.$ If $N=0,$  then $w=w_{Q_0}$ is a depth zero valuation and result holds trivially. Assume now that $N\geq 1.$ Then by Remark \ref{1.1.3} (i), $\Delta=\bigcup_{j=0}^{N}\Delta_j,$ where  $\Delta_j=\{j\}\cup\{\vartheta_{j}\}$ and $\vartheta_j$ is either  empty or an ordered set without a maximal element. Since for each $i\in\Delta,$ $Q_i$ is an ABKP for $w,$ so $w_{Q_i}$ is a valuation on $K(X)$ and we  denote it by $w_i.$

			Suppose first that $\vartheta_j=\emptyset$ for some $0\leq j\leq N,$ then $Q_{j+1}$ is not a limit key polynomial, i.e., $Q_{j+1}\in\psi(Q_j)$ which in view of  Proposition \ref{2.1.6} (v), implies that  $Q_{j+1}$ and $Q_j$ are key polynomials for $w_j$ and 
		$$w_{Q_{j+1}}(=w_{j+1})=[w_j; Q_{j+1}, w_{Q_{j+1}}(Q_{j+1})=w(Q_{j+1})]$$ is an ordinary augmentation of $w_j$ and $Q_{j+1},$ i.e., $w_j\rightarrow w_{j+1}$ is an ordinary augmentation. In fact $Q_j$ is a key polynomial of minimal degree for $w_j,$ i.e., $\deg Q_j=\deg(w_j)$ and as $\Phi_{w_j,w_{j+1}}=[Q_{j+1}]_{w_j},$  so
		$$\deg w_j<\deg Q_{j+1}=\deg(\Phi_{w_j,w_{j+1}}).$$ 
		
		Assume now that $\vartheta_j\neq \emptyset$ for some $0\leq j\leq N.$ 
		Then by Remark \ref{1.1.3}, for each $i\in\vartheta_j$ there exists an ABKP,  $Q_i$ for $w$ such that  $Q_i\in\psi(Q_j)$ and $\deg Q_i=\deg Q_j=m_j ~\text{(say)},$ where $Q_j$ is the ABKP corresponding to  $\{j\}.$  From Proposition \ref{2.1.6} (v), it follows that each $Q_i$ is a key polynomial for $w_j$ 
		and $w_{Q_i}$ is an ordinary augmentation of $w_j$ with respect to $w(Q_i),$ i.e.,
		$$w_{Q_i}(=w_i)=[w_{j};Q_i, w(Q_i)].$$ Therefore, for each $i\in\vartheta_{j},$ $w_i$ is an ordinary augmentation of $w_j.$
		Now for any $i<i'$ in $\vartheta_{j},$ since  $\delta(Q_i)<\delta(Q_{i'}),$ so by Proposition \ref{2.1.6} (ii), we have that $w_i(Q_{i'})<w(Q_{i'}),$ which together with $\deg Q_i=\deg Q_{i'}$ implies that $Q_{i'}\in\psi(Q_i).$  On using Proposition \ref{2.1.6} (v), we get that 
		$Q_{i'}$ is a key polynomial for $w_i,$ and $$w_{i'}=[w_i; Q_{i'}, w(Q_{i'})],$$ i.e., $w_{i'}$ is an ordinary augmentation of $w_i$ with respect to $w(Q_{i'})$  for every $i<i'\in\vartheta_{j}.$ From Corollary \ref{1.1.16} and Remark \ref{1.1.17},  we have  $\psi(Q_i)=\Phi_{w_i,w}=\Phi_{w_i, w_{i'}},$ and as $Q_i\notin\psi(Q_i),$ so 
		\begin{align}\label{1.25}
			w_i(Q_{i})=w_{i'}(Q_i),~ \forall ~ i'>i~ \text{in}~\vartheta_{j},
		\end{align}
		which in view of  Theorem \ref{1.1.19}, implies that $Q_{i'}\not\sim_{w_i} Q_i.$
		Hence $\mathcal{W}_j:=\{w_i\}_{i\in\vartheta_j}$ is a continuous family of augmentations such that for each $i\in\vartheta_{j},$ $w_i$ is an ordinary augmentation of $w_j$ with respect to  $w(Q_i)$ and from (\ref{1.25}),  $Q_i$ is $\mathcal{W}_j$-stable  with stability degree $m_j.$ As $\vartheta_{j}\neq\emptyset,$ so  by Definition \ref{1.1.22}, $Q_{j+1}$ is a limit key polynomial.  Since  $\delta(Q_i)<\delta(Q_{i'}),$ for every $i<i'\in\vartheta_{j},$ and    $w_{i'}(Q_{j+1})<w(Q_{j+1})$ for every $i'\in\vartheta_j,$ (because  $i'<j+1\in\Delta$), so  in view of Proposition \ref{2.1.6} (iv), we have   that $$w_i(Q_{j+1})<w_{i'}(Q_{j+1})~ \text{for every $i<i'\in \vartheta_{j}$}.$$  In fact  $\deg Q_{j+1}$ is minimal with this property, because $Q_{j+1}$ is a limit key polynomial, and therefore $Q_{j+1}$ is an unstable polynomial of minimal degree, say, ${m_{j}}_\infty,$ i.e., $Q_{j+1}$
		is a MLV limit key polynomial for $\mathcal{W}_j.$ Now since $\deg Q_j<\deg Q_{j+1},$ so  $m_j<{m_{j}}_\infty$ and hence
		$\mathcal{W}_j=\{w_i\}_{i\in\vartheta_j}$ is an essential continuous family of augmentations.
		
		Let $\gamma_{j+1}$ denotes the valuation $w(Q_{j+1}).$  Clearly, $\gamma_{j+1}>w_i(Q_{j+1}),$ for all  $i\in\vartheta_{j}.$ Let  $f$ be any polynomial in $K[X]$ with $Q_{j+1}$-expansion  $\sum_{s\geq 0} f_s Q_{j+1}^s.$  As $\deg f_s<\deg Q_{j+1}={m_{j}}_\infty,$ so all coefficients are $\mathcal{W}_j$-stable, i.e., $w_i(f_s)=w_{i'}(f_s)$ for all $i'>i$ in $\vartheta_{j}$ and we denote these stable values by   $\rho_{\mathcal{W}_j}(f_s).$ Then  	$$\rho_{j+1}(f)=\min_{s\geq 0}\{\rho_{\mathcal{W}_j} (f_s)+s\gamma_{j+1}\}$$ is a valuation on $K[X],$ which implies that  $\rho_{j+1}=[\mathcal{W}_j; Q_{j+1}, \gamma_{j+1}]$ is a limit augmentation of  an essential continuous family of augmentations of $w_j,$ (or $w_j\rightarrow \rho_{j+1}$ is a limit augmentation). Therefore $\rho_{j+1}(Q_{j+1})=\gamma_{j+1}=w(Q_{j+1}),$ i.e., $Q_{j+1}\notin \Phi_{\rho_{j+1},w}.$ 
		Since $Q_{j+1}$ is a key polynomial of minimal degree for $\rho_{j+1},$ so $\deg Q_{j+1}=\deg(\rho_{j+1})$ and hence in view of Theorem  \ref{2.1.2} we have that $$\rho_{j+1}=w_{Q_{j+1}}~ (=w_{j+1}).$$
		As  $w_j<w_{j+1}\leq w,$ so if $w_{j+1}=w,$ then by Remark \ref{1.1.17},  $\Phi_{w_j,w_{j+1}}=\Phi_{w_j,w}=\psi(Q_j),$ otherwise this equality holds in view of Corollary \ref{1.1.16} (ii) and  Remark \ref{1.1.17}. By Remark \ref{1.1.18} (ii), we  have that $\deg (w_j)=\deg Q_j=\deg (\psi(Q_j))=\deg( \Phi_{w_j,w_{j+1}}).$ In fact, $Q_{j+1}\notin\Phi_{w_j,w_{j+1}}=\psi(Q_j),$ because  $\deg Q_{j}<\deg Q_{j+1}.$ 

 Clearly, $w_{Q_N}=w,$  for if there exist some polynomial $f\in K[X]$ such that $w_{Q_N}(f)<w(f),$ then as $\Lambda$ is a complete set, so $w_{Q_i}(f)=w(f)$ for some $0\leq i<N.$ But this will imply that $w(f)=w_{Q_i}(f)\leq w_{Q_N}(f)<w(f).$
		
		Thus from the above arguments it follows that $$w_0\xrightarrow{Q_1,\gamma_1} w_1\xrightarrow{Q_2,\gamma_2}\cdots \longrightarrow w_{N-1}\xrightarrow{Q_N,\gamma_N} w_N=w,$$  where $w_i=w_{Q_i},$ $\gamma_i=w(Q_i)$ for every $0\leq i\leq N,$ is a MLV chain  whose last valuation is $w_N=w,$ such that:
		\begin{itemize}
			\item if $\vartheta_{j}= \emptyset,$ then $w_j\longrightarrow w_{j+1}$ is an ordinary augmentation.
			\item if $\vartheta_{j}\neq \emptyset,$ then $w_j\longrightarrow w_{j+1}$ is a limit augmentation of an essential continuous family of augmentations of $w_j.$
		\end{itemize}
		Finally the chain is complete because $w_0=w_{Q_0},$ where $Q_0=X,$ is  defined by the pair $(0,w_0(X))$ and is a depth zero valuation.
		\end{proof}

	\begin{proof}[Proof of Theorem \ref{1.1.12}]
		Let $$	w_0\xrightarrow{\phi_1,\gamma_1} w_1\xrightarrow{\phi_2,\gamma_2}\cdots \longrightarrow w_{N-1}\xrightarrow{\phi_N,\gamma_N} w_N=w,$$ be a complete finite MLV chain of $w.$ If $N=0,$  then $w_0=w$ is a depth zero valuation and the result holds trivially. Assume now that $N\geq 1.$
			Then each $\phi_j$ is a key polynomial for $w_j$ of minimal degree, i.e., $\deg \phi_j=\deg(w_j),$ and  $w_j(\phi_{j})=w(\phi_j).$   Therefore $\phi_j\notin\Phi_{w_j,w}$ which in view of Theorem \ref{2.1.2}, implies that $\phi_j$ is an ABKP for $w$ and
			\begin{align}\label{1.26}
			 w_j=w_{\phi_j}~\text{for every $0\leq j\leq N.$ } 
		\end{align}
		Suppose first that $w_j\longrightarrow w_{j+1}$ is an ordinary augmentation. Then by definition of MLV chain of $w,$ we have that $\phi_{j+1}\in \Phi_{w_j,w_{j+1}}.$ As $w_j<w_{j+1}\leq w,$ so if $w_{j+1}<w,$ then by Corollary \ref{1.1.16},  $\Phi_{w_j, w_{j+1}}=\Phi_{w_j,w},$ i.e., $w_j(\phi_{j+1})<w(\phi_{j+1}),$ otherwise this holds trivially.  Now from (\ref{1.26}), we get that $w_{\phi_j}(\phi_{j+1})<w(\phi_{j+1})$ which together with the minimality of $\deg \phi_{j+1},$ implies that $\phi_{j+1}\in \psi(\phi_j)$ and hence from Lemma \ref{1.2.6}, it follows that $$\delta(\phi_j)<\delta(\phi_{j+1}).$$
		
		Assume now that $w_j\longrightarrow w_{j+1}$ is a limit augmentation. Then $\phi_{j+1}$  is a MLV limit key polynomial for an essential continuous family (say) $\mathcal{W}_j$ of augmentations of $w_j.$ 
		Let $\mathcal{W}_j=\{\rho_i\}_{i\in \mathbf{A}_j},$  where $\mathbf{A}_j$ is some totally ordered set without a maximal element and for each $i\in\mathbf{A}_j,$  $\rho_i=[w_j, \phi_i, \gamma_{i}]$ is an ordinary augmentation of $w_j$ with stability degree, (say) $m_j=\deg \phi_j=\deg \phi_i.$ Also, for all $i<i'\in\mathbf{A}_j,$ $\phi_{i'}$ is a key polynomial for $\rho_i$ such that $$\phi_{i'}\not\thicksim_{\rho_i}\phi_i,~ \text{i.e.,}~ \phi_{i'}\nmid_{\rho_i}\phi_i~\text{and}~ \rho_{i'}=[\rho_i; \phi_{i'}, \gamma_{i'}].$$  Since $\rho_i<w$ and $\phi_{i'}\nmid_{\rho_i} \phi_i,$ so by Theorem \ref{1.1.15}, $\rho_i(\phi_i)=w(\phi_i),$ which together with Corollary \ref{1.1.16}, gives 
		\begin{align}\label{1.24}
			\phi_i\notin\Phi_{\rho_i,w}=\Phi_{\rho_i,\rho_{i'}}=[\phi_{i'}]_{\rho_i}.
		\end{align}
	Now by Remark \ref{1.1.18} (ii), for each $i\in \mathbf{A}_j,$ $\phi_i$ is a key polynomial for $\rho_i$ of minimal degree, i.e.,  $\deg \phi_i=\deg\rho_i,$  therefore keeping in mind that $\rho_i<w,$ equation (\ref{1.24}) in view of Theorem \ref{2.1.2} (ii),  implies that each $\phi_i$ is an ABKP for $w$ and $$\rho_i=w_{\phi_i},~\text{for all $i\in\mathbf{A}_j$}.$$  Hence for each $i<i'\in\mathbf{A}_j,$ $\phi_i$ and $\phi_{i'}$ are ABKPs for $w$ such that $$w(\phi_i)=w_{\phi_i}(\phi_i)<w_{\phi_{i'}}(\phi_{i'})=w(\phi_{i'})~ \text{and}~  \deg \phi_i=\deg \phi_{i'},$$ which in view of Proposition \ref{2.1.6} (iii), implies that $w_{\phi_i}(\phi_{i'})<w(\phi_{i'}).$ Therefore, by Lemma \ref{1.2.6}, we get that  $$\phi_{i'}\in\psi(\phi_i)~\text{and}~ \delta(\phi_i)<\delta(\phi_{i'})~\text{for every $i<i'\in \mathbf{A}_{j}$}.$$ As $\deg \phi_j=\deg \phi_i,$ for every $i\in\mathbf{A}_{j}$ and $w(\phi_j)<w(\phi_i),$ so again by Proposition \ref{2.1.6} (iii), we have that   $\phi_i\in\psi(\phi_j)$  and then  $$\delta(\phi_j)<\delta(\phi_i)~\forall ~i\in\mathbf{A}_j,$$ follows from Lemma \ref{1.2.6}.   Since $\mathcal{W}_j$ is essential, so $\deg \phi_{j}=\deg \phi_i<\deg \phi_{j+1},$ for every $i\in\mathbf{A}_j,$ this together with the fact that $\phi_{j+1}$ is an ABKP for $w,$ implies that $$\delta(\phi_{j})<\delta(\phi_{j+1}), ~ \delta(\phi_i)<\delta(\phi_{j+1}) ~\text{and}~ \phi_{j+1}\notin \psi(\phi_j).$$
	
	For every $0\leq j\leq N,$ let $\Delta_j=\{j\}\cup\mathbf{A}_j,$  and $\Delta=\bigcup_{j=0}^{N}\Delta_j.$ We now show that  $\Lambda=\{\phi_i\}_{i\in\Delta}$ is a complete set of ABKPs for $w.$ Clearly, as shown above for every $i<i'\in\Delta,$ we have $\delta(\phi_i)<\delta(\phi_{i'}).$ Therefore  the set $\Lambda$ is well-ordered with respect to the ordering given by $\phi_i< \phi_{i'}$ if and only if  $\delta(\phi_i)<\delta(\phi_{i'})$ for every $i<i'\in \Delta.$ 
	 It only remains  to prove that for any polynomial $f\in K[X],$ there exist some $i\in\Delta$ such that $w_{\phi_i}(f)=w(f).$
	 If $\deg f<\deg \phi_i$ for some $i\in \Delta,$ then $w_{\phi_i}(f)=w(f).$  On the other hand, if $\deg f\geq \deg \phi_i$ for all $i\in\Delta,$ then using the fact that  $w_N=w$ and  (\ref{1.26}),  we get  $w_{\phi_N}=w.$ Hence $w_{\phi_N}(f)=w(f).$
	 Thus 
	 $\{\phi_i\}_{i\in\Delta}$ is a complete set of ABKPs for $w$ such that 
	 \begin{itemize}
	 	\item if $w_j\longrightarrow w_{j+1}$ is an ordinary augmentation, then $\phi_{j+1}\in\psi(\phi_j),$
	 	\item if $w_j\longrightarrow w_{j+1}$ is a limit augmentation, then $\phi_{j+1}\notin\psi(\phi_j)$ and therefore, $\mathbf{A}_j\neq\emptyset$ implies that $\phi_{j+1}$ is a limit key polynomial.
	 \end{itemize}
	\end{proof}

	\section*{Acknowledgement}
	Research of the first author is supported by CSIR (grant no.\  09/045(1747)/2019-EMR-I).

\end{document}